\documentclass[reqno,11pt]{amsart}

\usepackage[all]{xy}
\CompileMatrices
\usepackage{amssymb}
\usepackage{mathrsfs}
\usepackage{cite}
\usepackage{amsfonts}
\usepackage[active]{srcltx}

\renewcommand{\to}{\longrightarrow}

\newcommand{\p}{\varphi}

\newtheorem{Thm}{Theorem}[section]
\newtheorem{Prop}[Thm]{Proposition}

\newtheorem{Lemma}[Thm]{Lemma}
{\theoremstyle{definition}
}
{\theoremstyle{remark}
}

{\theoremstyle{remark}
}

{\theoremstyle{remark}
}
{\theoremstyle{remark}
}
{\theoremstyle{remark}
}

\numberwithin{equation}{section}

\title[The algebra of the Catalan monoid as an incidence algebra]{The algebra of the Catalan monoid as an incidence algebra: a simple proof}

\author{Stuart Margolis\ and Benjamin Steinberg}
\address{Department of Mathematics\\
Bar Ilan University\\ 52900 Ramat Gan \\ Israel \and Department of Mathematics\\
    City College of New York\\
    Convent Avenue at 138th Street\\
    New York, New York 10031\\
    USA}
\thanks{}
\email{margolis@math.biu.ac.il\and bsteinberg@ccny.cuny.edu}
\date{\today}


\begin{document}
\begin{abstract}
We give a direct straightforward proof that there is an isomorphism between the algebra of the Catalan monoid $C_n$, that is, the monoid of all order-preserving, weakly increasing self-maps $f$ of $[n]=\{1,\ldots, n\}$, over any commutative ring with identity and the incidence algebra of a certain poset over the ring.
\end{abstract}
\maketitle

\section{Introduction}

The \emph{Catalan monoid} $C_n$ is the monoid of all order-preserving, weakly increasing self-maps $f$ of $[n]=\{1,\ldots, n\}$.  It is well known that the cardinality of $C_n$ is the $n$-th Catalan number. See for example, \cite[Ex.6.19(s)]{Stanley2}. $C_n$ has appeared in many guises in combinatorics and combinatorial semigroup theory under different names. For example, it has also been called the monoid of non-decreasing parking
functions or, simply, the monoid of order-decreasing and order-preserving functions \cite{Catalancombo, Catalanmonoid, GM}.

 It was first observed by Hivert and Thi\'ery \cite{HivThiery}, via an indirect proof, that $\Bbbk C_n$ is isomorphic to the incidence algebra of a certain poset $P_{n-1}$ for any field $\Bbbk$; a different indirect proof using the representation theory of $\mathscr J$-trivial monoids can be found in the second author's book~\cite{BenBook}.  Grensig gave a direct isomorphism from the incidence algebra of $P_{n-1}$ to $\Bbbk C_n$ for any base commutative ring with unit $\Bbbk$, but her proof is quite involved and long because of the technical recursive construction of a complete set of orthogonal idempotents.

Here we show that there is a straightforward direct isomorphism from $\Bbbk C_n$ to the incidence algebra of $P_{n-1}$ over any base commutative ring with unit whose details are trivial to check.  The proof is similar to that used by the second author for inverse monoids~\cite{mobius1} and Stein for more general monoids \cite{steinpartial, Ehresmann, EhresmannErrata}.  The complication in previous approaches is avoided here as we show that the isomorphism is given by a unipotent upper triangular $0/1$-matrix.

\section{Combinatorics of the Catalan monoid}

If $n\geq 0$ is an integer, let $P_n$ be the poset consisting of subsets of $[n]$ ordered by $X\leq Y$ if and only if $|X|=|Y|$ and if $X=\{x_1<\cdots<x_k\}$ and $Y=\{y_1<\cdots<y_k\}$, then $x_i\leq y_i$ for $i=1,\ldots, k$. We shall also need the following refinement of this order given by $X\preceq Y$ if $|X|<|Y|$ or if $|X|=|Y|$ and $X\leq Y$.


There is a well-known bijection between the Catalan monoid $C_{n+1}$ and the set of ordered pairs $(X,Y)$ of elements of $P_n$ with $X\leq Y$ given as follows.  If $X\leq Y$, define $f_{X,Y}\colon [n+1]\to [n+1]$ in $C_{n+1}$ by
\[f_{X,Y}(i) = \begin{cases}y_1, &\text{if}\ 1\leq i\leq x_1\\
                            y_j, &\text{if}\ x_{j-1}<i\leq x_j\\
                            n+1, &\text{if}\ i>x_k \end{cases}\]
where we have retained the notation of the previous paragraph for $X$ and $Y$. The condition that $X\leq Y$ guarantees that $f_{X,Y}$ is order preserving and weakly increasing. Conversely, if $f\in C_{n+1}$, let $Y=f([n+1])\setminus \{n+1\}$ and let $X=\{i\in f^{-1}(Y)\mid i=\max\{f^{-1}(f(i))\}$; so $X$ consists of the maximum element of each partition block of $f$ except the block of $n+1$.  Then it is straightforward to check from $f$ being order-preserving and weakly increasing that $X\leq Y$ and that $f=f_{X,Y}$.  See~\cite[Ch. 17 Sec. 5]{BenBook} for details where we note that a slightly different convention was used.

Let us say that $S\subseteq [n]$ is a \emph{partial cross-section} for $f\in C_{n+1}$ if $n+1\notin f(S)$ and $f|_S$ is injective, i.e., $|S|=|f(S)|$.
We denote by $\mathsf{PCS}(f)$ the set of partial cross-sections of $f$.  The following proposition is straightforward from the definitions and so we omit the proof.

\begin{Prop}\label{p:basic.facts}
Let $S$ be a partial cross-section for $f=f_{X,Y}$ in $C_{n+1}$.
\begin{enumerate}
\item $S\preceq X$.
\item $S\leq f(S)$.
\item $f(S)\subseteq Y$, hence $f(S)\preceq Y$.
\item $X\in \mathsf{PCS}(f)$.
\end{enumerate}
\end{Prop}

The following lemma is key to our proof.

\begin{Lemma}\label{l:product}
Let $f,g\in C_{n+1}$.  Then $S\in \mathsf{PCS}(fg)$ if and only if $S\in \mathsf{PCS}(g)$ and $g(S)\in \mathsf{PCS}(f)$.
\end{Lemma}
\begin{proof}
Suppose first that $S\in \mathsf{PCS}(fg)$.  Since $f(n+1)=n+1$ and $n+1\notin fg(S)$, it follows that $n+1\notin g(S)$.  Also, $fg|_S$ is injective and hence $g|_S$ is injective.  Thus $S$ is a partial cross-section for $g$.  Since $|fg(S)|=|S|=|g(S)|$ we also have that $f|_{g(S)}$ is injective.  As $n+1\notin fg(S)$, we see that $g(S)$ is a partial cross-section for $f$.  Conversely, assume that $S$ is a partial cross-section for $g$ and $g(S)$ is a partial cross-section for $f$.  Then $n+1\notin fg(S)$ and $|fg(S)|=|g(S)|=|S|$.  Thus $S$ is a partial cross-section for $fg$.
\end{proof}

\section{The isomorphism of algebras}
Let $\Bbbk$ be a commutative ring with unit and let $n\geq 0$.  Let $I(P_n,\Bbbk)$ be the incidence algebra of $P_n$ over $\Bbbk$.  It can be viewed as the $\Bbbk$-algebra with basis all ordered pairs $(X,Y)$ with $X\leq Y$ in $P_n$ and where the product is defined on basis elements by
\[(U,V)(X,Y) = \begin{cases}(X,V), & \text{if}\ Y=U\\ 0, & \text{else.}\end{cases}\]  In other words, this is the algebra of the category corresponding to the poset $P_n$.

We partially order the basis $C_{n+1}$ of $\Bbbk C_{n+1}$ by saying $f_{U,V}$ comes before $f_{X,Y}$ if $U\preceq X$ and $V\preceq Y$ and, similarly, we order the basis of $I(P_n,\Bbbk)$ by saying that $(U,V)$ comes before $(X,Y)$ if $U\preceq X$ and $V\preceq Y$.  We can now prove the main result.

\begin{Thm}\label{t:main}
There is an isomorphism $\p\colon \Bbbk C_{n+1}\to I(P_n,\Bbbk)$ of $\Bbbk$-algebras given by
\[\p(f) = \sum_{S\in \mathsf{PCS}(f)} (S,f(S))\] for $f\in \Bbb C_{n+1}$.
\end{Thm}
\begin{proof}
It is immediate from Proposition~\ref{p:basic.facts} that $\p$ is well defined and that the matrix of $\p$ as a homomorphism of free $\Bbbk$-modules with respect to our preferred bases and orderings is unipotent upper triangular.  Indeed, $\p(f_{X,Y}) = (X,Y)+a$ where $a$ is a sum of certain terms $(U,V)$ with $U\prec X$ and $V\preceq Y$.  It follows that $\p$ is an isomorphism of $\Bbbk$-modules.  It remains to show that it is a ring homomorphism.  Indeed,
\begin{align*}
\p(f)\p(g) &= \sum_{T\in \mathsf{PCS}(f)}(T,f(T))\cdot \sum_{S\in \mathsf{PCS}(g)}(S,g(S))\\
&= \sum_{S\in \mathsf{PCS}(g), g(S)\in \mathsf{PCS}(f)}(S,f(g(S)))\\
&=\p(fg)
\end{align*}
where the last equality is by Lemma~\ref{l:product}.  This completes the proof.
\end{proof}


\bibliography{stubib}
\bibliographystyle{abbrv}
\end{document}